\documentclass[preprint,12pt]{elsarticle}

\usepackage{amssymb}

\usepackage[noend]{algpseudocode}
\usepackage{algorithmicx,algorithm}

\newtheorem{Theorem}{\quad Theorem}[section]

\newtheorem{Corollary}[Theorem]{\quad Corollary}
\newtheorem{Lemma}[Theorem]{\quad Lemma}

\newenvironment{proof}{\noindent{\bf Proof.}\ \ }{$\Box$\medskip}

\journal{Discrete Applied Mathematics}

\begin{document}

\begin{frontmatter}



\title{Linear algorithms on Steiner domination of trees\tnoteref{label1}}
\tnotetext[label1]{The research is
supported by Natural Science Foundation of China (No.61173002,11701543). }

\author[add1]{Yueming Shen}
\author[add1,add2]{Chengye Zhao}
\ead{cyzhao@cjlu.edu.cn}
\author[add2]{Chenglin Gao}
\author[add2]{Yunfang Tang}

\address[add1]{Science of Economics and Management,China Jiliang University,Hangzhou,China}
\address[add2]{Science of College,China Jiliang University,Hangzhou,China}

\begin{abstract}
A set of vertices $W$ in a connected graph $G$ is called a Steiner dominating set if $W$ is both Steiner and dominating set. The Steiner domination number $\gamma_{st}(G)$ is the minimum cardinality of a Steiner dominating set of $G$. A linear algorithm is proposed in this paper for finding a minimum Steiner dominating set for a tree $T$.
\end{abstract}

\begin{keyword}
linear algorithm \sep Steiner dominating set \sep Steiner domination number


\end{keyword}

\end{frontmatter}


\section{Introduction}

In this paper, we only consider finite, connected and undirected graph $G$. We refer to the books \cite{West2001,HHS1998} for notation and terminology on graph theory and theory of domination.

Let $G = (V(G), E(G))$ be a graph with the order of vertex set $|V(G)|$ and the order of edge set $ |E(G)|$. The open neighborhood and the closed neighborhood of a vertex $v \in V$ are denoted by $N(v) = \{ {u \in V(G):vu \in E(G)} \}$ and $N[v] = N(v) \cup \{v\} $, respectively. For a vertex set $S \in V(G)$, $N(S) = \bigcup\limits_{v \in S} {N(v)} $, and $N[S] = \bigcup\limits_{v \in S} {N[V]} $. The distance $d(u,v)$ between two vertices $u$ and $v$ of a connected graph $G$ is the length of shortest $u-v$ path in $G$. For a non-empty set $W$ of vertices in connected graph $G$, the Steiner distance $d(W)$ of $W$ is the minimum size of a connected subgraph of $G$ containing $W$. Obviously, each such subgraph is a tree and is called a Steiner tree or a Steiner $W$-tree. The set of all vertices of $G$ that lie on some Steiner $W$-tree is denoted by $S(W)$. If $S(W) = V(G)$ then $W$ is called Steiner set of $G$. The Steiner number $s(G)$ is the minimum cardinality of a Steiner set.

Chartrand and Zhang introduced the concept of Steiner number of a connected graph $G$ in \cite{CZ2002}. Pelayo corrected main result in \cite{Pelayo2004}. He proved that not all Steiner sets are geodetic sets and there are connected graphs whose Steiner number is strictly lower than their geodetic number. Hernando et al. \cite{HJMPS2005} have studied the relationships between Steiner sets and geodetic sets and between Steiner sets and monophonic sets. Many results on Steiner distance were given in \cite{COTZ1989,SJ2011}.

A subset $S$ of V(G) is called dominating set if every vertex $v \in V$ is either a vertex of $S$ or is adjacent to a vertex of $S$. The domination number $\gamma(G)$ is the minimum cardinality of minimal dominating set of $G$. A systematic visit of each vertex of a tree is called a tree traversal. A set of vertices $W$ in a connected graph $G$ is called a Steiner dominating set if $W$ is both Steiner and dominating set. The Steiner domination number $\gamma_{st}(G)$ is the minimum cardinality of a Steiner dominating set of $G$.

The concept of Steiner domination was introduced in \cite{JEA2013}, and Vaidya etc. have obtained various results on Steiner domination numbers in \cite{VK20151,VK20152,VK2018}.

The most algorithmic complexity of domination and related parameters of graphs are NP-complete or NP-hard problems. But there are many linear algorithms for domination and related parameters in trees, such as domination, total domination and secure domination in trees \cite{CGH1975,LPHH1984,BVV2014}. In this paper, we present a linear algorithm of Steiner domination in trees. It is similar to an algorithm due to Mitchell, Cockayne and Hedetniemi \cite{MCH1979} for computing the domination number of an arbitrary tree.



\section{Lemmas}

 A vertex of a graph $G$ is called a leaf or end-vertex if it is adjacent to only one vertex in $G$. A vertex $v$ is an extreme vertex if the subgraph induced by its neighbors is complete. Thus, every end-vertex is an extreme vertex.

\begin{Lemma}\label{Lem1}\cite{CZ2002} Each extreme vertex of a graph G belongs to every Steiner set of G. In particular, each end-vertex of G belongs to every Steiner set of G. \end{Lemma}

The following corollary is an immediate consequence of Lemma \ref{Lem1}.

\begin{Corollary}\label{Cor1}\cite{CZ2002} Every nontrivial tree with exactly k end-vertices has Steiner number k. \end{Corollary}

By Corollary \ref{Cor1} and Lemma \ref{Lem1}, we have

\begin{Corollary}\label{Cor2} Let L(T) include all end-vertices of a tree T, then L(T) is a Steiner set of T. \end{Corollary}

Let $H = T[V-N[L(T)]]$ be the induced subgraph of $T$ from the set $V-N[L(T)]$. We have,

\begin{Theorem}\label{The1} For any nontrivial tree T, $\gamma_{st}(T) = |L(T)| + \gamma(H)$. \end{Theorem}
\begin{proof} Let $S$ be a minimum dominating set of $H$ and $\gamma(H)=|S|$. By Corollary \ref{Cor2}, L(T) is  a Steiner set of T. Hence the set $S \cup L(T)$ is a Steiner dominating set of T and $\gamma_{st}(T) \le |L(T)| + \gamma(H)$.

Nextly, we prove $\gamma_{st}(T) \ge |L(T)| + \gamma(H)$. By contradiction, let $\gamma_{st}(T) < |L(T)| + \gamma(H)$ and there is a $\gamma_{st}$-set $S'$ such that $\gamma_{st}(T)=|S'|$. By Lemma \ref{Lem1}, L(T) is a subset of each minimum Steiner set of $T$. Let $S''=S'-L(T)$. By the definition of $H$, $S''$ is  a minimum dominating set of $H$ such $|S''| = \gamma_{st}(T) - |L(T)| < \gamma(H)$, it is a contradiction.
\end{proof}
\section{Linear algorithm for Forest Domination }

In this section, we construct a linear algorithms for domination in forest. The algorithms is based on the algorithm for computing the domination number of an arbitrary tree by Mitchell, Cockayne and Hedetniemi \cite{MCH1979}.

By Theorem \ref{The1}, the minimum Steinier dominating set of a tree is divided two subsets: $L(T)$ and the $\gamma$-set of subgraph $H$ of $T$.

By the definition of $H$, $H$ is a tree or a forest. So the algorithm in \cite{MCH1979} has to be changed for computing the domination number of a forest.
Algorithm 1 for domination of a forest $F$, and each tree $T$ in $F$ is rooted. Two linear arrays are maintained during this traversal process:

Parent[$i$]:contains the index of the parent of vertex $i$ in a forest $F$; in the Parent array, that the Parent of a vertex labelled $i$ is given by Parent[$i$], and Parent[$j$]=0 if vertex $j$ is the root of a tree in $F$; for any vertex labelled $i$ in $F$, Parent[$i$]$<i$.

Label[$i$]:contains three states:'Bound','Required' and 'Free';  the usage of Label array is similar to the algorithm in \cite{MCH1979}.

Compared with the algorithm in \cite{MCH1979}, we add the condition that Parent[$i$] $\neq$ 0. This condition ensures  that we construct the dominating set of each tree in $F$ by Algorithm 1 and get the minimum dominating set of a forest $F$.
\begin{algorithm}
\caption{Forest Domination} 
\hspace*{0.02in} {\bf Input:} 
input parameters a forest $F$ represented by an array Parent[1..$N$]\\
\hspace*{0.02in} {\bf Output:} 
output a minimum dominating set $D$ of $F$
\begin{algorithmic}[1]
\State $D \leftarrow \varnothing$ 
\For{$i$=1 to $N$} 
　　\State Label[$i$]='Bound'
\EndFor
\For{$i$=$N$ to 1 by -1}
　　\If{Label[$i$]='Bound' and Parent[$i$]$\neq$ 0} 
　　　　\State Label[Parent[$i$]]='Required'
　　\Else \If{Label[$i$]='Required'}
　　　　      \State  $D \leftarrow D \cup \{i\}$
              \If{Label[Parent[$i$]]='Bound'}
                  \State Label[Parent[$i$]]='Free'
              \EndIf
         \EndIf
　　\EndIf
\EndFor
\For{$i$=1 to $N$}
    \If{Parent[$i$]=0 and (Label[$i$]='Bound' or Label[$i$]='Required')}
        \State $D \leftarrow D \cup \{i\}$
    \EndIf
\EndFor
\end{algorithmic}
\end{algorithm}

\begin{Theorem}\label{The2} (Complexity of Algorithm 1). If the input forest to Algorithm 1 has order n, then both the space complexity and the
worst-case time complexity of Algorithm 1 are O(n).\end{Theorem}
\begin{proof} Setp 1 can be performed in O(1) time. Steps 2-3, 4-11, 12-14 are three for-loops, and each operation in these loops can be performed in O(1) time. So the total operation time is $3n+1 = O(n)$.

A total of $3n$ memory units are required to store the array Label,Parent and the set $D$. Two memory units are
required to store the values of the variables $i$ and $N$. The space complexity of Algorithm 1 is therefore $3n + 2 =
O(n)$. \end{proof}

\section{Linear algorithm for Tree Steiner Domination }

In this section, we construct a linear algorithms for Steiner domination in a tree. By Theorem \ref{The1}, the definition of $H$ and Algorithm 1, we only consider the structures of $L(T)$ and $H$. Five linear arrays are maintained during this traversal process:

Parent[$i$]:contains the index of the parent of vertex $i$ in tree $T$; in the Parent array, that the Parent of a vertex labelled $i$ is given by Parent[$i$], and Parent[$i$]=0 if vertex $i$ is the root of $T$; for any vertex labelled $i$ in $T$, Parent[$i$]$<i$.

Flag[$i$]:Flag[$i$]=0 if the vertex $i$ is a end-vertex of $T$, else Flag[$i$]=1.

PFlag[$i$]:PFlag[$i$]=1 if the vertex $i$ is adjacent to a end-vertex of $T$, else PFlag[$i$]=0.

Index[$i$]:contains the index in $T$ of the vertex $i$ in $H$.

NParent[$i$]:contains the index of the parent of vertex $i$ in a forest $H$; in the Parent array, that the Parent of a vertex labelled $i$ is given by Parent[$i$], and Parent[$j$]=0 if vertex $j$ is the root of a tree in $H$; for any vertex labelled $i$ in $H$, Parent[$i$]$<i$.

By the steps 1-23 in Algorithm 2, we get $L(T)$ (the end-vertex set of $T$) and NParent array of $H=G[V-N[L(T)]]$. We obtain the $\gamma$-set of $H$ by the step 24 in Algorithm 2 (Nparent array as a input of Algorithm 1). Finally, we have a minimum Steiner dominating set of tree $T$ by the step 25 in Algorithm 2.

We conclude this section with a result on the space and time complexities of Algorithm 2.

\begin{algorithm}
\caption{Tree Steiner Domination} 
\hspace*{0.02in} {\bf Input:} 
input parameters a tree $T$ represented by an array Parent[1..$N$]\\
\hspace*{0.02in} {\bf Output:} 
output a minimum Steiner dominating set $SD$ of $T$
\begin{algorithmic}[1]
\State $SD \leftarrow \varnothing$ 
\For{$i$=1 to $N$} 
　　\State Flag[$i$]=0
    \State PFlag[$i$]=0
\EndFor
\For{$i$=1 to $N$} 
    \If{Parent[$i$] $\neq$ 0}
        \State Flag[Parent[$i$]]=1
    \EndIf
\EndFor
\For{$i$=1 to $N$} 
    \If{Flag[$i$] = 0}
        \State PFlag[Parent[$i$]]=1
    \EndIf
\EndFor
\For{$i$=1 to $N$} 
    \State $m=0$
    \If{Flag[$i$] = 0}
        \State $SD \leftarrow SD \cup \{i\}$
    \Else
        \If{PFlag[$i$] $\neq$ 1}
            \State $m=m+1$
            \State Index[$m$]=$i$
        \EndIf
    \EndIf
\EndFor
\For{$i$=1 to $m$} 
    \If{PFlag[Parent[Index[$i$]]] = 0}
        \State NParent[$i$]=Parent[Index[$i$]]
    \Else
        \State NParent[$i$]=0
    \EndIf
\EndFor
\State Input NParent as Parent into Algorithm 1, and get the result $D$
\State $SD \leftarrow SD \cup D$
\end{algorithmic}
\end{algorithm}

\begin{Theorem}\label{The3} (Complexity of Algorithm 2). If the input tree to Algorithm 2 has order n, then both the space complexity and the
worst-case time complexity of Algorithm 2 are O(n).\end{Theorem}
\begin{proof} Setps 1 and 25 can be performed in O(1) time. Steps 2-4, 5-7, 8-10, 11-18, 19-23 are five for-loops, and each operation in these loops can be performed in O(1) time. So the total operation time of these loops is $4n+m$. The operation time in step 24 is $O(m)$ by Theorem \ref{The2}. So the total operation time is $4n+m+2+O(m) = O(n)$.

A total of $8n$ memory units are required to store the array Label, Parent, NParent, Flag, PFlage, Index, the set $D$ and $SD$. Three memory units are
required to store the values of the variables $i$ , $N$ and $m$. The space complexity of Algorithm 2 is therefore $8n + 3 =
O(n)$. \end{proof}

\end{document}